\documentclass[11pt,a4paper,twoside]{article}
\usepackage{amsthm, amsfonts,amsmath}

\newtheorem{theorem}{Theorem}[section]
\newtheorem{corollary}{Corollary}[section]
\newtheorem{definition}{Definition}[section]
\newtheorem{example}{Example}[section]
\newtheorem{lemma}{Lemma}[section]
\newtheorem{proposition}{Proposition}[section]
\newtheorem{remark}{Remark}[section]

\def\Ric{\operatorname{Ric}}
\def\tr{\operatorname{trace\,}}
\def\eq{\hspace*{-2.1mm}&=&\hspace*{-2.1mm}}

\author{Vladimir Rovenski\footnote{Department of Mathematics, University of Haifa,
 Israel
       \newline e-mail: {\tt vrovenski@univ.haifa.ac.il}
       } }

\title{Generalized Ricci solitons and Einstein metrics on weak $K$-contact manifolds}

\begin{document}

\date{}

\maketitle

\begin{abstract}
We study metric structures on a smooth manifold (introduced in our recent works \cite{RovP-arxiv,RWo-2} and called a weak contact metric structure and a weak $K$-contact structure) which generalize the metric contact and $K$-contact structures, and allow a new look at the classical theory. First, we characterize weak $K$-contact manifolds among all weak contact metric mani\-folds by the property well known for $K$-contact manifolds, and find when a Riemannian manifold endowed with a unit Killing vector field forms a weak $K$-contact structure.
Second, we find sufficient conditions for a weak $K$-contact manifold with parallel Ricci tensor or with a generalized Ricci soliton structure to be an Einstein manifold.


\end{abstract}




\section{Introduction}
\label{sec:00}

The growing interest in contact geometry is associated with its important role in mechanics in explaining physical phenomena.
In addition, many recent articles have been motivated by the question of how interesting self-similar solutions of the Ricci flow equation,
i.e., Ricci solitons, can be for contact metric geometry. Some of them find conditions when a contact
manifold equipped with a Ricci-type soliton structure carries a canonical (e.g., Einstein or constant curvature) metric, e.g.~\cite{Blairsurvey,G-D-2020,CZB-2022,N-R-2016}.

$K$-contact manifolds (i.e., contact metric manifolds whose Reeb vector field is Killing)
have been studied by several geometers, e.g., \cite{blair2010riemannian,YK-1985},
and it is seen that the $K$-contact structure is intermediate between the contact and Sasakian structures.
The Reeb vector field $\xi$ of the $K$-contact structure is a unit Killing vector field, and the influence
of constant length Killing vector fields on the Riemannian geometry has been studied by several authors
from different points of view, e.g.,~\cite{B-N-2008,D-B-2021,N-2021}.
An interesting result related to the above question is that a $K$-contact
manifold equipped with generalized Ricci soliton structure has an Einstein metric,~e.g.,~\cite{G-D-2020}.

In \cite{RWo-2}, we introduced the ``weakened" metric structures on a smooth manifold
(replacing the complex structure on the characte\-ristic distribution with a nonsingular skew-symmetric tensor).
They generalize the metric contact, $K$-contact, Sasakian and cosymplectic structures, and allow a new look at the classical theory.
 In \cite{RWo-2}, we build retraction of weak structures with positive partial Ricci curvature onto the set of classical structures.
 In~\cite{RovP-arxiv} (where the definition of weak structures is a bit more general than in \cite{RWo-2}),
 we proved that the~{weak Sasakian structure} is {weak $K$-contact};
and a weak almost contact metric manifold is weak Sasakian if and only if it is a Sasakian manifold.
 In this article we study weak $K$-contact manifolds using their
sectional and Ricci curvature in the $\xi$-direction.
Our goal is to show that the weak $K$-contact structure can be a useful tool for studying unit Killing vector fields on Riemannian manifolds,
and that some results for $K$-contact manifolds can be extended to the case of weak $K$-contact manifolds.
For example, we answer the question of when a weak $K$-contact manifold carries a generalized Ricci soliton structure or just an Einstein metric.

The article is organized as follows.
In Section~\ref{sec:01}, following the introductory Section~\ref{sec:00}, we recall basics of weak contact metric manifolds.
Next three sections contain the main results.
In Section~\ref{sec:02}, we characterize (in Theorem~\ref{T-3.1}) weak $K$-contact manifolds among all weakly contact metric manifolds by the property $\varphi=-\nabla\xi$ (well known for $K$-contact manifolds), and find (in Theorem~\ref{prop2.1b}) when a Riemannian manifold endowed with a unit Killing vector field forms a weak $K$-contact structure.
In~Section~\ref{sec:03}, for a weak $K$-contact manifold, we calculate (in Proposition~\ref{thm6.2D}) the Ricci curvature in the $\xi$-direction, then find (in Theorem~\ref{T-4.1}) sufficient condition for such a manifold with parallel Ricci tensor to be an Einstein manifold.
In Section~\ref{sec:04}, we find (in Theorem~\ref{T-5.1}) sufficient conditions for a weak $K$-contact manifold admitting a generalized Ricci soliton structure to be an Einstein manifold.

\section{Preliminaries}
\label{sec:01}


Here, we recall basics of some metric structures that generalize the almost contact metric structure (see \cite{RWo-2}, or, equivalently, \cite{RovP-arxiv} with $\nu=1$).
A~\textit{weak almost contact structure} on a smooth manifold $M^{2n+1}$ is a set $(\varphi,Q,\xi,\eta)$,
where $\varphi$ is a rank-(1,1) tensor, $Q$ is a nonsingular $(1,1)$-tensor, $\xi$ is the Reeb vector field and $\eta$ is a dual 1-form,
i.e., $\eta(\xi)=1$, satisfying
\begin{equation}\label{2.1}
 \varphi^2 = -Q + \eta\otimes \xi, \qquad
 Q\,\xi = \xi.
\end{equation}
The form $\eta$ determines a smooth $2n$-dimensional distribution ${\cal D}:=\ker\eta$,
the collection of subspaces ${\cal D}_m=\{{X}\in T_m M: \eta({X})=0\}$ for $m\in M$.
We assume that ${\cal D}$ is $\varphi$-invariant,
\begin{equation}\label{2.1-D}
 \varphi {X}\in{\cal D},\quad {X}\in{\cal D},
\end{equation}
as in the theory of almost contact structure \cite{blair2010riemannian,YK-1985}, where $Q={\rm id}_{\,TM}$.
By \eqref{2.1} and \eqref{2.1-D}, our distribution ${\cal D}$ is invariant for $Q$: $Q({\cal D})={\cal D}$.
 If~there is a Riemannian metric $g$ on $M$ such that
\begin{align}\label{2.2}
 g(\varphi {X},\varphi {Y})= g({X},Q\,{Y}) -\eta({X})\,\eta({Y}),\quad {X},{Y}\in\mathfrak{X}_M,
\end{align}
then $(\varphi,Q,\xi,\eta,g)$ is called a {\it weak almost contact metric structure} on $M$, and $g$
is called a \textit{compatible} metric.
A weak almost contact manifold $M(\varphi,Q,\xi,\eta)$ endowed with a compatible Riemannian metric $g$
is called a \textit{weak almost contact metric manifold} and is denoted by $M(\varphi,Q,\xi,\eta,g)$.
Some sufficient conditions for the existence of a compatible metric on a weak almost contact manifold are given in~\cite{RWo-2}.

Putting ${Y}=\xi$ in \eqref{2.2} and using $Q\,\xi=\xi$, we get, as in the classical theory,
 $\eta({X})=g(\xi,{X})$.
In~particular, $\xi$ is $g$-orthogonal to ${\cal D}$ for any compatible metric $g$.

By \eqref{2.2}, we get $g({X},Q\,{X})=g(\varphi {X},\varphi {X})>0$ for any nonzero vector ${X}\in{\cal D}$, thus $Q$ is positive definite.
For a weak almost contact structure on a smooth manifold $M$, the tensor $\varphi$ has rank $2n$ and
\[
 \varphi\,\xi=0,\quad \eta\circ\varphi=0,\quad \eta\circ Q=\eta,\quad [Q,\,\varphi]=0;
\]
moreover, for a weak almost contact metric structure, $\varphi$ is skew-symmetric and $Q$ is self-adjoint.

A \textit{weak contact metric structure} is defined as a weak almost contact metric structure satisfying $\Phi=d\eta$
(thus, $d\Phi=0$), where $\Phi({X},{Y})=g({X},\varphi {Y})\ ({X},{Y}\in\mathfrak{X}_M)$ is called the {fundamental $2$-form}, and
\begin{align}\label{3.3A}
 d\eta({X},{Y}) &= \frac12\,\{{X}(\eta({Y})) - {Y}(\eta({X})) - \eta([{X},{Y}])\},\quad {X},{Y}\in\mathfrak{X}_M, \\
\label{3.3}
 d\Phi({X},{Y},Z) &= \frac{1}{3}\,\big\{ {X}\,\Phi({Y},Z) + {Y}\,\Phi(Z,{X}) + Z\,\Phi({X},{Y}) \notag\\
 &-\Phi([{X},{Y}],Z) - \Phi([Z,{X}],{Y}) - \Phi([{Y},Z],{X})\big\},\quad {X},{Y},Z\in\mathfrak{X}_M.
\end{align}

\begin{remark}\rm
A differential $k$-\textit{form} on a manifold $M$ is a skew-symmetric tensor field
$\omega$ of  type $(0, k)$. According the conventions of e.g.,
\cite{KN-69},
the formula
\begin{eqnarray}\label{eq:extdiff}
\nonumber
 & d\omega ({X}_1, \ldots , {X}_{k+1}) = \frac1{k+1}\sum\nolimits_{\,i=1}^{k+1} (-1)^{i+1} {X}_i(\omega({X}_1, \ldots , \hat{{X}}_i\ldots, {X}_{k+1}))\\
 & +\sum\nolimits_{\,i<j}(-1)^{i+j}\omega ([{X}_i, {X}_j], {X}_1, \ldots,\hat{{X}}_i,\ldots,\hat{{X}}_j, \ldots, {X}_{k+1}),
\end{eqnarray}
where ${X}_1,\ldots, {X}_{k+1}\in\mathfrak{X}_M$ and $\,\hat{\cdot}\,$ denotes the
operator of omission, defines a $(k+1)$-form $d\omega$ called the \textit{exterior differential} of $\omega$.
Note that \eqref{3.3A} and \eqref{3.3}, e.g., \cite{blair2010riemannian}, correspond to \eqref{eq:extdiff} with $k=1$ and $k=2$.
\end{remark}

For a weak contact metric structure, the distribution ${\cal D}$ is non-integrable (has no integral hypersurfaces),
since $g([{X}, \varphi {X}], \xi)= 2\,d\eta(\varphi {X},{X}) = g(\varphi {X},\varphi {X})>0$ for any nonzero ${X}\in{\cal D}$.

 The Nijenhuis torsion $[\varphi,\varphi]$ of $\varphi$ is given~by
\begin{align*}
 [\varphi,\varphi]({X},{Y}) = \varphi^2 [{X},{Y}] + [\varphi {X}, \varphi {Y}] - \varphi[\varphi {X},{Y}] - \varphi[{X},\varphi {Y}],\quad {X},{Y}\in\mathfrak{X}_M.
\end{align*}
A weak almost contact structure $(\varphi,Q,\xi,\eta)$ is called {\it normal} if the following tensor is zero:
\begin{align}\label{2.6{X}}
 N^{\,(1)}({X},{Y}) = [\varphi,\varphi]({X},{Y}) + 2\,d\eta({X},{Y})\,\xi,\quad {X},{Y}\in\mathfrak{X}_M.
\end{align}
The following tensors $N^{\,(2)}, N^{\,(3)}$ and $N^{\,(4)}$ are well known in the classical theory, see \cite{blair2010riemannian,YK-1985}:
\begin{align*}
 N^{\,(2)}({X},{Y}) &= (\pounds_{\varphi {X}}\,\eta)({Y}) - (\pounds_{\varphi {Y}}\,\eta)({X})
 \overset{\eqref{3.3A}}=2\,d\eta(\varphi {X},{Y}) - 2\,d\eta(\varphi {Y},{X}) ,  \\
 N^{\,(3)}({X}) &= (\pounds_{\,\xi}\,\varphi){X} = [\xi, \varphi {X}] - \varphi [\xi, {X}],\\
 N^{\,(4)}({X}) &= (\pounds_{\,\xi}\,\eta)({X}) = \xi(\eta({X})) - \eta([\xi, {X}])
 \overset{\eqref{3.3A}}= 2\,d\eta(\xi, {X}).
\end{align*}

\begin{remark}\rm
Let $M(\varphi,Q,\xi,\eta)$ be a weak almost contact manifold.
Consider the product manifold $\bar M = M\times\mathbb{R}$,
where $\mathbb{R}$ has the Euclidean basis $\partial_t$,
and define tensor fields $\bar\varphi$ and $\bar Q$ on $\bar M$ putting
\begin{align*}
 \bar\varphi({X},\, a\,\partial_t) &= (\varphi {X} - a\,\xi,\,\eta({X})\,\partial_t),\\
 \bar Q({X},\, a\,\partial_t) &= (Q{X},\,a\,\partial_t) ,
\end{align*}
where $a\in C^\infty(M)$.
Thus, $\bar\varphi({X},0)=(\varphi {X},0)$, $\bar Q({X},0)=(Q{X},0)$ for ${X}\in\ker\varphi$,
$\bar\varphi(\xi,0)=(0,\partial_t)$, $\bar Q(\xi,0)=(\xi,0)$
and
$\bar\varphi(0,\partial_t)=(-\xi,0)$, $\bar Q(0,\partial_t)=(0,\partial_t)$.
By the above, $\bar\varphi^{\ 2}=-\bar Q$.
The tensors $N^{\,(i)}\ (i=1,2,3,4)$ appear when we derive the integrability condition $[\bar\varphi, \bar\varphi]=0$
(vanishing of the Nijenhuis torsion of $\bar\varphi$\,) and express the normality condition $N^{\,(1)}=0$ of a
$(\varphi,Q,\xi,\eta)$
on $M$.
\end{remark}

\begin{theorem}[see \cite{RovP-arxiv} with $\nu=1$]
{\rm a)} For a weak almost contact metric structure $(\varphi,Q,\xi,\eta,g)$, the vanishing of $N^{\,(1)}$ implies that $N^{\,(3)}$ and $N^{\,(4)}$ vanish and
 $N^{\,(2)}({X},{Y}) =\eta([\tilde Q {X},\,\varphi {Y}])$.

{\rm b)} For a weak contact metric manifold, the tensors $N^{\,(2)}$ and $N^{\,(4)}$ vanish and
the trajectories of $\xi$ are geodesics, i.e., $\nabla_\xi\,\xi=0$; moreover,
$N^{\,(3)}\equiv0$ if and only if $\,\xi$ is a Killing vector field.
\end{theorem}

\begin{definition}[see \cite{RovP-arxiv} with $\nu=1$]\rm
Two weak almost contact structures $(\varphi,Q,\xi,\eta)$ and $(\varphi',Q',\xi,\eta)$ on $M$ are said to be homothetic if
the following is valid for some real $\lambda>0$:
\begin{subequations}
\begin{align}\label{Tran'}
 & \varphi = \sqrt\lambda\ \varphi', \\
\label{E-Q'-lambda}
 & Q\,|_{\,{\mathcal D}}=\lambda\,Q'|_{\,\mathcal D} .
\end{align}
Two weak contact metric structures $(\varphi,Q,\xi,\eta,g)$ and $(\varphi',Q',\xi,\eta,g')$ on $M$
are said to be homothetic if they satisfy conditions (\ref{Tran'},b) and
\begin{align}\label{Tran2'}
 g|_{\,{\mathcal D}} = \lambda^{\,-\frac12}\,g'|_{\,{\mathcal D}},\quad
 g(\xi,\,\cdot) = {g}'(\xi,\,\cdot) .
\end{align}
\end{subequations}
\end{definition}

\begin{lemma}[see \cite{RovP-arxiv} with $\nu=1$]\label{P-22}
Let $(\varphi,Q,\xi,\eta)$ be a weak almost contact structure such that
\begin{subequations}
\begin{align*}
 Q\,|_{\,{\mathcal D}}=\lambda\,{\rm id}_{\mathcal D},
\end{align*}
for some real $\lambda>0$. Then the following is true:

\noindent
$\bullet$ $(\varphi', \xi, \eta)$ is an almost contact structure, where $\varphi'$ is given by
\eqref{Tran'}.

\noindent
$\bullet$ If $(\varphi,Q,\xi,\eta,g)$ is a weak contact metric structure,
and $\varphi',g'$ satisfy {\rm (\ref{Tran'},c)},
then $(\varphi',\xi,\eta,{g}')$ is a contact metric structure.
\end{subequations}
\end{lemma}

A \textit{weak $K$-contact manifold} is defined as a weak contact metric manifold whose Reeb vector field $\xi$
is Killing (or, infinitesimal isometry, e.g., \cite{KN-69}), i.e.,
\begin{align}\label{E-nabla}
 (\pounds_{\,\xi}\,g)({X},{Y}):= \xi(g({X},{Y})) -g([\xi,{X}],{Y}) - g({X},[\xi,{Y}]) = g(\nabla_{X}\,\xi, {Y}) +g(\nabla_{Y}\,\xi, {X})=0 .
\end{align}
Here $\pounds_{\,\xi}$ is the Lie derivative in the $\xi$-direction and $\nabla$ is the Levi-Civita connection.
A normal weak contact metric manifold is called a {\it weak Sasakian manifold}.
 Recall that a weak Sasakian structure is weak $K$-contact, see \cite[Proposition~4.1]{RovP-arxiv}, furthermore,
 a weak almost contact metric structure
 is weak Sasakian if and only if
it is homothetic to a Sasakian structure, see \cite[Theorem~4.1]{RovP-arxiv}.

The relationships between the different classes of weak structures (considered in this article) can be summarizes in the diagram (well known in the case of classical structures):
\[
 \left|
   \begin{array}{c}
     weak \\
     almost\ contact \\
   \end{array}
 \right|
\overset{metric}\longrightarrow
\left|
   \begin{array}{c}
     weak\ almost \\
     contact\ metric\\
   \end{array}
 \right|
\overset{\Phi=d\eta}\longrightarrow
  \left|
   \begin{array}{c}
     weak \\
     contact\ metric \\
   \end{array}
 \right|
\overset{\xi\,-\,Killing}\longrightarrow
  \left|
   \begin{array}{c}
     weak \\
     K-contact \\
   \end{array}
 \right|.
\]

\smallskip

A ``small" (1,1)-tensor $\tilde{Q} = Q - {\rm id}$ is
a measure of the difference between a weakly contact structure and a contact one, and
$\tilde Q=0$ means the classical contact geometry. Note that $[\tilde{Q},\varphi]=0$ and $\tilde{Q}\,\xi=0$.

\begin{lemma}[\cite{RovP-arxiv} with $\nu=1$]
For a weak contact metric manifold, we get
\begin{equation}\label{3.1}
 g((\nabla_{{X}}\varphi){Y},Z) = \frac12 g(N^{\,(1)}({Y},Z),\varphi {X})
 +g(\varphi {X}, \varphi {Y})\,\eta(Z) -g(\varphi {X}, \varphi Z)\,\eta({Y})
 + \frac12 N^{\,(5)}({X},{Y},Z),
\end{equation}
where the skew-symmetric with respect to ${Y}$ and $Z$ tensor $N^{\,(5)}({X},{Y},Z)$
supplements the sequence of tensors $N^{\,(i)}\ (i=1,2,3,4)$
and for a weak contact metric manifold is given by
\begin{align*}
 & N^{\,(5)}({X},{Y},Z) = (\varphi Z)\,(g({X}, \tilde Q{Y})) -(\varphi {Y})\,(g({X}, \tilde QZ)) \\
 & +\,g([{X}, \varphi Z], \tilde Q{Y}) - g([{X},\varphi {Y}], \tilde QZ)
   +g([{Y},\varphi Z] -[Z, \varphi {Y}] - \varphi[{Y},Z],\ \tilde Q {X}).
\end{align*}
\end{lemma}

\begin{remark}\label{R-02}\rm
For  a contact metric structure,  \eqref{3.1} (with $N^{\,(5)}=0$) gives the result in \cite[Corollary~6.1]{blair2010riemannian}.
Note that only one new tensor $N^{\,(5)}$ is needed for further study of a weak contact metric structure.
In~particular, by \eqref{3.1}, we get
$g((\nabla_{\xi}\,\varphi){Y}, Z) = \frac12\,N^{\,(5)}(\xi,{Y},\,Z)$ and
\begin{align}\label{KK}
\nonumber
 N^{\,(5)}({X},\xi,Z) & = g(
 N^{\,(3)}(Z),\, \tilde Q {X}),\\
\nonumber
 N^{\,(5)}(\xi,{Y},Z) &= g([\xi, \varphi Z], \tilde Q{Y}) -g([\xi,\varphi {Y}], \tilde QZ),\\
 N^{\,(5)}(\xi,\xi,Z) &= N^{\,(5)}(\xi,{Y},\xi)=0.
\end{align}
\end{remark}

\section{Unit Killing vector fields}
\label{sec:02}

\begin{proposition}
On a weak $K$-contact manifold, we get $N^{\,(1)}(\xi,\,\cdot)=0$ and
\begin{eqnarray}
\label{E-31}
 N^{\,(5)}(\xi,\,\cdot\,,\,\cdot) & = & N^{\,(5)}(\,\cdot\,,\,\xi,\,\cdot) =0,\\
 \label{E-31A}
 \pounds_{\,\xi}\,{\tilde Q} &=& \nabla_\xi\,{\tilde Q} = 0,\\
\label{E-30-phi}
 \nabla_{\xi}\,\varphi &=& 0.
\end{eqnarray}
\end{proposition}

\begin{proof}
By
\eqref{2.6{X}} and $d\eta(\xi,\,\cdot) = \Phi(\xi,\,\cdot) = 0$
we get
\[
 N^{\,(1)}(\xi,{X})
 = [\varphi,\varphi]({X},\xi) = \varphi^2 [{X},\xi] - \varphi[\varphi {X},\xi] = \varphi N^{\,(3)}({X})=0.
\]
 By \cite[Lemma~3.1]{RovP-arxiv}  with $\nu=1$ and $h=\frac12\,N^{\,(3)}=0$, we get $N^{\,(5)}(\xi,\,\cdot\,,\,\cdot) = 0$, $\pounds_{\,\xi}{\tilde Q} =0$ and
\begin{equation*}
 g(Q\,\nabla_{{X}}\,\xi, Z) = g(\varphi Z,Q{X}) - \frac 12\,N^{\,(5)}({X},\xi,\varphi Z).
\end{equation*}
By \eqref{KK}$_1$ with $N^{\,(3)}=0$, we get $N^{\,(5)}(\,\cdot\,, \xi,\,\cdot)=0$.
We use $[\varphi, {\tilde Q}]=0$ to obtain $\nabla_\xi\,{\tilde Q}=0$:
\[
  (\pounds_{\,\xi}\,{\tilde Q}){X} = [\xi, {\tilde Q}{X}] - {\tilde Q}[\xi, {X}] = (\nabla_\xi{\tilde Q}){X} +[\varphi, {\tilde Q}]{X}
 =(\nabla_\xi\,{\tilde Q}){X}.
\]
This completes the proof of \eqref{E-31} and \eqref{E-31A}.
Next, from \eqref{3.1} with ${X}=\xi$ we get \eqref{E-30-phi}.
\end{proof}

In the next theorem, we characterize weak $K$-contact manifolds among all weak contact metric manifolds
by the following well known property of $K$-contact manifolds, see \cite{blair2010riemannian}:
\begin{equation}\label{E-30}
 \nabla\,\xi = -\varphi .
\end{equation}

\begin{theorem}\label{T-3.1}
A weak contact metric manifold is weak $K$-contact $($that is $\xi$ is a Killing vector field$)$ if and only if \eqref{E-30} is valid.
\end{theorem}

\begin{proof}
Let a weak contact metric manifold satisfy \eqref{E-30}. By skew-symmetry of $\varphi$, we get
$(\pounds_{\,\xi}\,g)({X},{Y})=g(\nabla_{X}\,\xi, {Y})+g(\nabla_{Y}\,\xi, {X})=-g(\varphi {X}, {Y})-g(\varphi {Y}, {X})=0$, thus, $\xi$ is a Killing vector~field.

Conversely, let our manifold be weak $K$-contact.
By \eqref{3.1} with ${Y}=\xi$, using $N^{\,(1)}(\xi,\,\cdot)=0$ and $N^{\,(5)}({X},\xi,Z)=0$, see \eqref{E-31}, we get
\begin{eqnarray*}
 g((\nabla_{{X}}\,\varphi)\,\xi,Z) \eq \frac12\,g(N^{\,(1)}(\xi,Z),\varphi {X}) - g(\varphi {X}, \varphi Z)+\frac12\,N^{\,(5)}({X},\xi,Z)
 = g( \varphi^2 {X}, Z).
\end{eqnarray*}
Hence, $(\nabla_{{X}}\,\varphi)\,\xi = \varphi^2 {X}$.
From this and $0=\nabla_{{X}}\,(\varphi\,\xi)=(\nabla_{{X}}\,\varphi)\,\xi+\varphi\nabla_{{X}}\,\xi$, we obtain
$\varphi(\nabla_{{X}}\,\xi+\varphi {X}) = 0$.
Since $\nabla_{{X}}\,\xi+\varphi {X}\in{\cal D}$
and $\varphi$ is invertible when restricted on ${\cal D}$,
we get that $\nabla_{{X}}\,\xi=-\varphi {X}$.
\end{proof}


If a plane contains $\xi$, then its sectional curvature is called $\xi$-\textit{sectional curvature}.
It is well known that the $\xi$-sectional curvature of a $K$-contact manifold~is constant equal to 1.
Recall that a Riemannian manifold with a unit Killing vector field and the property
$R_{{X},\,{\xi}}\,{\xi}
={X}\ ({X}\bot\,\xi)$ is a $K$-contact manifold, e.g., \cite[Theorem~3.1]{YK-1985} or \cite[Proposition~7.4]{blair2010riemannian}.
We~generalize this result in the following

\begin{theorem}\label{prop2.1b}
A Riemannian manifold $(M^{\,2n+1},g)$ admitting a unit Killing vector field $\xi$ with positive
$\xi$-sectional curvature
is a weak $K$-contact manifold $M(\varphi,Q,\xi,\eta,g)$ with
the following structural tensors: $\eta({X}) = g({X}, \xi)$,
$\varphi = -\nabla\,\xi$, see \eqref{E-30}, and $Q {X} = R_{{X},\,{\xi}}\,{\xi}$ for ${X}\in\ker\eta$.
\end{theorem}

\begin{proof}
Let $\eta = g(\,\cdot\,, \xi)$ and ${\cal D}=\ker\eta$.
Put $\varphi {X} = -\nabla_{X}\,\xi$ and $Q {X} = R_{{X},\,{\xi}}\,{\xi}$ for ${X}\in{\cal D}$.
Since $\xi$ is a Killing vector field, we obtain the property $d\eta=\Phi$ for $\Phi({X},{Y})=g({X},\varphi {Y})$:
\[
 d\eta({X},{Y}) = \frac12\,(g(\nabla_{X}\,\xi, {Y}) - g(\nabla_{Y}\,\xi, {X})) = - g(\nabla_{Y}\,\xi, {X}) = g({X},\varphi {Y}).
\]
Since $\xi$ is a unit Killing vector field, we get $\nabla_\xi\,\xi=0$ and
$\nabla_{X}\nabla_{Y}\,\xi - \nabla_{\nabla_{X}\,{Y}}\,\xi = R_{\,\xi,{X}}\,{Y}$. Thus, $\varphi\,\xi=0$ and
\[
 \varphi^2 {X} = \nabla_{\nabla_{X}\,\xi}\,\xi
=R_{\,\xi,{X}}\,\xi
 = -Q{X}\quad ({X}\in{\cal D}).
\]
 Put $Q\,\xi = \xi$. Therefore, \eqref{2.1} is valid, and $Q$ is positive definite,
that completes the proof.
\end{proof}

\begin{example}\rm
By Theorem~\ref{prop2.1b}, we can search for examples of weak $K$-contact (not $K$-contact) manifolds
among Riemannian manifolds of positive sectional curvature that admit unit Killing vector fields.
Indeed, let $M$ be a convex hypersurface (ellipsoid) with induced metric $g$ of the Euclidean space~$\mathbb{R}^{2n+2}$,
\[
 M = \Big\{(u_1,\ldots,u_{2n+2})\in\mathbb{R}^{2n+2}: \sum\nolimits_{\,i=1}^{n+1} u_i^2 + a\sum\nolimits_{\,i=n+2}^{2n+1} u_i^2 = 1\Big\},
\]
where $0<a=const\ne1$ and $n\ge1$ is odd. The sectional curvature of $(M,g)$ is positive. It follows that
\[
 \xi = (-u_2, u_1, \ldots , -u_{n+1}, u_{n}, -\sqrt a\,u_{n+3}, \sqrt a\,u_{n+2}, \ldots , -\sqrt a\,u_{2\,n+2}, \sqrt a\,u_{2\,n+1})
\]
is a Killing vector field on $\mathbb{R}^{2\,n+2}$, whose restriction to $M$ has unit length.
Since $\xi$ is tangent to
$M$ (i.e., $M$ is invariant under the flow of $\xi$), so $\xi$ is a unit Killing vector field on $(M,g)$, see \cite[p.~5]{D-B-2021}. For $n=1$, we get a weak $K$-contact manifold
$M^3=\big\{u_1^2 + u_2^2 + a u_3^2 + a u_4^2 = 1\big\}\subset\mathbb{R}^{4}$ with $\xi = (-u_2, u_1, -\sqrt a\,u_{4}, \sqrt a\,u_{3})$.

Other examples of weak $K$-contact (not $K$-contact) manifolds are obtained from \cite[Theorem~12]{B-N-2008}:
``On every sphere $S^{2n-1},\ n \ge 2$, for any $\varepsilon>0$ there exists a (real analytic) Riemannian metric $g$ of cohomogeneity 1
and a (real analytic) Killing vector field $\xi$ of unit length on $(S^{2n-1}, g)$ such that

1)~all sectional curvatures of $(S^{2n-1}, g)$ differ from 1 by at most $\varepsilon$;

2)~the vector field $\xi$ has both closed and non-closed integral trajectories".
\end{example}

\begin{corollary}
A weak $K$-contact structure $(\varphi,Q,\xi,\eta,g)$ with constant positive $\xi$-sectional curvature, $K(\xi,{X})=\lambda>0$ for some $\lambda=const\in\mathbb{R}$ and all ${X}\in{\cal D}$, is homothetic to a $K$-contact structure $(\varphi',\xi,\eta,g')$ after the transformation {\rm(\ref{Tran'}-c)}.
\end{corollary}

\begin{proof}
Note that $K(\xi,{X})=\lambda\ ({X}\in{\cal D})$ if and only if $R_{{X},\xi}\,\xi = \lambda\,{X}\ ({X}\in{\cal D})$.
By $Q{X}=R_{{X},\xi}\,\xi\ ({X}\in{\cal D})$, see Theorem~\ref{prop2.1b}, we get $Q {X} = \lambda {X}\ ({X}\in{\cal D})$.
By Lemma~\ref{P-22}(ii), $(\varphi',\xi,\eta,{g}')$ is a contact metric structure.
Using \eqref{E-nabla}, we get $(\pounds_{\,\xi}\,g')({X},{Y})=\lambda (\pounds_{\,\xi}\,g)({X},{Y})\ ({X},{Y}\in{\cal D})$
and $(\pounds_{\,\xi}\,g')(\xi,\,\cdot)=0$.
By $\pounds_{\,\xi}\,g=0$, we get $\pounds_{\,\xi}\,g'=0$; thus $(\varphi',\xi,\eta,g')$ is a $K$-contact structure.
\end{proof}

\section{The Ricci curvature in the characteristic direction}
\label{sec:03}


Denote by $R_{{X},{Y}}\,Z$ the curvature tensor and by $\Ric^\sharp$ the Ricci operator of $g$ associated with the Ricci tensor ${\rm Ric}$ and given by
 ${\rm Ric}({X},{Y})=g(\Ric^\sharp {X},{Y})$ for all ${X},{Y}\in\mathfrak{X}_M$.
The~Ricci curvature in the $\xi$-direction is given by ${\rm Ric}({\xi},{\xi})=\sum_{\,i=1}^{\,2n} g(R_{e_i,\,{\xi}}\,{\xi}, e_i)$,
where $(e_i)$ is any local orthonormal basis of~${\cal D}$.

In the next proposition, we generalize three particular properties of $K$-contact manifolds to weak $K$-contact manifolds.

\begin{proposition}\label{thm6.2D}
For a weak $K$-contact manifold, the following equalities are true:
\begin{eqnarray}
\label{E-R0}
 && R_{{\xi},\,{X}}\,{Y} = (\nabla_{X}\,\varphi){Y}, \\
\label{E-R1}
 && R_{{X},\,{\xi}}\,{\xi} =-\varphi^2 {X} ,\\
\label{E-R1b}
 && {\rm Ric}({\xi},{\xi}) = \tr Q = 2\,n + \tr\tilde Q .
\end{eqnarray}
\end{proposition}

\begin{proof}
Using \eqref{E-30}, we derive
\begin{eqnarray}\label{Eq-R}
\nonumber
 R_{Z,\,{X}}\,{\xi} \eq \nabla_Z(\nabla_{X}\,{\xi}) -\nabla_{X}(\nabla_Z\,{\xi}) -\nabla_{[Z,{X}]}\,{\xi}) \\
 \eq \nabla_{X} (\varphi Z) - \nabla_Z (\varphi {X}) + \varphi([Z,{X}])
 = (\nabla_{X}\,\varphi)Z -(\nabla_Z\,\varphi){X}.
\end{eqnarray}
Note that $(\nabla_{X}\,\Phi)({Y},Z) =g((\nabla_{X}\,\varphi)Z, {Y}) =-g((\nabla_{X}\,\varphi){Y}, Z)$.
Using condition $d\Phi=d^2\eta=0$, we get
\begin{equation}\label{E-dPhi-three}
 (\nabla_{X}\,\Phi)({Y},Z)+(\nabla_{Y}\,\Phi)(Z,{X})+(\nabla_Z\,\Phi)({X},{Y})=0.
\end{equation}
From \eqref{Eq-R}, using \eqref{E-dPhi-three}
and skew-symmetry of $\Phi$, we get \eqref{E-R0}:
\begin{eqnarray*}
 g(R_{{\xi},\,{X}}\,{Y} , Z) \eq g(R_{{Y},\,Z}\,{\xi}, {X})
 \overset{\eqref{Eq-R}}= (\nabla_Z\,\Phi)({X},{Y}) + (\nabla_{Y}\,\Phi)(Z,{X}) \\
 \hspace*{-2.1mm}&\overset{\eqref{E-dPhi-three}}=&\hspace*{-2.1mm} -(\nabla_{X}\,\Phi)({Y},Z)
 = g((\nabla_{X}\,\varphi){Y},Z).
\end{eqnarray*}
By \eqref{E-R0} with ${Y}=\xi$, using $\varphi\,\xi =0$ and
\eqref{E-30}, we find
\begin{eqnarray*}
 R_{{\xi}, {X}}\,{\xi} = (\nabla_{X}\,\varphi)\,\xi = -\varphi\nabla_{X}\, {\xi} = \varphi^2 {X} .
\end{eqnarray*}
This and \eqref{2.1}$_1$ yield \eqref{E-R1}.
By this, for any local orthonormal basis $(e_i)$ of ${\cal D}$, we get
\[
 {\rm Ric}({\xi},{\xi})
 \overset{\eqref{E-R1}}= -\sum\nolimits_{\,i=1}^{\,2n} g(\varphi^2 e_i, e_i)
 \overset{\eqref{2.1}}= \sum\nolimits_{\,i=1}^{\,2n} g(Q e_i, e_i).
\]
By the above and the equality $\tr Q = 2\,n + \tr\tilde Q$, \eqref{E-R1b} is true.
\end{proof}

Note that if a Riemannian manifold 
admits a unit Killing vector field $\xi$, then
$K(\xi,{X})\ge0\ ({X}\,\perp\,\xi,\ {X}\ne0)$, thus ${\rm Ric}({\xi},{\xi})\ge0$; moreover, ${\rm Ric}({\xi},{\xi})\equiv0$ if and only if $\xi$ is parallel: $\nabla\xi\equiv0$, for example, \cite{N-2021}.
In the case of $K$-contact manifolds, $K(\xi,{X})=1$, see \cite[Theorem~7.2]{blair2010riemannian}.

\begin{corollary}\label{Cor_K-positive}
For a weak $K$-contact manifold, the $\xi$-sectional curvature is positive:
\begin{equation}\label{E-Kmix}
 K(\xi,{X})=g(Q{X},{X})>0\quad ({X}\in{\cal D},\ \|{X}\|=1),
\end{equation}
therefore, for the Ricci curvature we get ${\rm Ric}({\xi},{\xi})>0$.
\end{corollary}

\begin{proof}
For any unit vector ${X}\in{\cal D}$, by \eqref{E-R1}, we get
\[
 0< g(\varphi {X}, \varphi {X}) = -g(\varphi^2 {X}, {X}) = g(Q {X}, {X}),
\]
thus $K(\xi,{X})>0$ and $\tr Q> 0$.
Therefore, from \eqref{E-R1b} we get the statement.
\end{proof}

 By Theorem~\ref{prop2.1b} and Corollary~\ref{Cor_K-positive}, using \cite[Corollary~3]{RWo-2}, we conclude the following.

\begin{corollary}
A~weak $K$-contact manifold $M(\varphi,Q,\xi,\eta,g_0)$ admits a smooth family of metrics $g_t\ (t\in\mathbb{R})$,
such that $M(\varphi_t,Q_t,\xi,\eta,g_t)$ are weak $K$-contact manifolds with certainly defined $\varphi_t$ and $Q_t$;
moreover, $g_t$ converges exponentially fast, as $t\to-\infty$, to a limit metric $\hat g$ that gives a $K$-contact structure.
\end{corollary}

The following theorem generalizes a well known result, e.g., \cite[Proposition 5.1]{YK-1985}.

\begin{theorem}\label{T-4.1}
A weak $K$-contact manifold with conditions $(\nabla\Ric)(\xi,\,\cdot)=0$
$($in particular, the Ricci tensor is parallel$)$ and $\tr Q=const$ is an Einstein manifold of scalar curvature
$(2\,n + 1)\tr Q$.
\end{theorem}

\begin{proof}
Differentiating \eqref{E-R1b} and using \eqref{E-30} and the conditions, we get
\[
 0 = \nabla_{Y}\,({\rm Ric}({\xi},{\xi})) = (\nabla_{Y}\,{\rm Ric})({\xi},{\xi}) +2\,{\rm Ric}(\nabla_{Y}\,{\xi},{\xi}))
 = -2\,{\rm Ric}(\varphi {Y},{\xi})),
\]
hence ${\rm Ric}({Y},{\xi}) = \eta({Y})\,{\rm Ric}(\xi,\xi) = \eta({Y})\tr Q$.
Differentiating this, then using
\[
 {X}(\eta({Y}))=g(\nabla_{X}\xi, {Y})=-g(\varphi {X},{Y})+g(\nabla_{X} {Y}, \xi)
\]
and assuming $\nabla_{X} {Y}=0$ at $x\in M$, gives
\[
 (\tr Q)\,g(\varphi {Y}, {X}) = \nabla_{X}\,({\rm Ric}({Y},{\xi}))
 = (\nabla_{X}\,{\rm Ric})({Y},{\xi}) +2\,{\rm Ric}({Y}, \nabla_{X}\,{\xi}) = -2\,{\rm Ric}({Y}, \varphi {X}),
\]
hence ${\rm Ric}({Y}, \varphi {X}) = (\tr Q)\,g({Y}, \varphi {X})$.
Therefore, we obtain
\[
 {\rm Ric}({X},{Y}) = (\tr Q)\,g({X},{Y})
 \]
for any vector fields ${X}$ and ${Y}$ on $M$, which means that $(M,g)$ is an Einstein manifold.
Using the definition of scalar curvature, $\tau = \tr\Ric$, we find $\tau=(2\,n + 1)\tr Q$.
\end{proof}

\begin{remark}\rm
For a weak $K$-contact manifold, by \eqref{E-R0} and $\nabla_{\xi}\,\varphi=\frac12\,N^{\,(5)}(\xi,{Y},\,Z)=0$, see Remark~\ref{R-02} and \eqref{E-31}, we get the following equality (well known for $K$-contact manifolds, e.g., \cite{blair2010riemannian}):
\[
 \Ric^\sharp(\xi)=\sum\nolimits_{\,i=1}^{\,2n}(\nabla_{e_i}\,\varphi)\,e_i,
\]
where $(e_i)$ is any local orthonormal basis of~${\cal D}$.
For $K$-contact manifolds, this gives $\Ric^\sharp(\xi)=2\,n\,\xi$, see \cite[Proposition~7.2]{blair2010riemannian},
and $\Ric(\xi,\xi)=2\,n$; moreover, the last condition characterizes $K$-contact manifolds among all contact metric manifolds.
%
\end{remark}

\section{Generalized Ricci solitons on weak $K$-contact manifolds}
\label{sec:04}

The \textit{generalized Ricci soliton} equation in a Riemannian manifold $(M, g)$ is defined by \cite{N-R-2016},
\begin{equation}\label{E-g-r-e}
 \frac12\,\pounds_{{X}}\,g = -c_1 {X}^\flat\otimes {X}^\flat + c_2\Ric + \lambda\,g
\end{equation}
for some smooth vector field ${X}$ and real $c_1, c_2$ and $\lambda$.
If ${X} = \nabla f$ in \eqref{E-g-r-e} for some $f\in C^\infty(M)$, then by the definition ${\rm Hess}_f({X},{Y})=\frac12\,(\pounds_{\nabla f}\,g)$({X},{Y}),
we get the \textit{generalized gradient Ricci soliton} equation
\begin{equation}\label{E-gg-r-e}
 {\rm Hess}_f = -c_1 df\otimes df + c_2\Ric + \lambda\,g .
\end{equation}
Each equation above is a generalization of the Einstein metric, $\Ric + \lambda\,g=0$.
For different values of $c_1,\,c_2$ and $\lambda$, equation \eqref{E-g-r-e} is a generalization of
Killing equation ($c_1 = c_2 = \lambda = 0$), equation for homotheties ($c_1 = c_2 = 0$), Ricci soliton equation
($c_1 = 0,\,c_2 = -1$), vacuum near-horizon geometry equation ($c_1 = 1,\ c_2 = 1/2$), e.g., \cite{G-D-2020}.

First, we formulate some lemmas.

\begin{lemma}\label{L-5.1}
For a weak $K$-contact manifold we get
\[
 (\pounds_{\,\xi}(\pounds_{{X}}\,g))({Y},\xi) = g({X},{Y}) + g(\nabla_\xi\nabla_\xi\,{X}, {Y}) + {Y} g(\nabla_\xi\,{X}, \xi)
\]
for all smooth vector fields ${X},{Y}$ with ${Y}$ orthogonal to $\xi$.
\end{lemma}

\begin{proof} This uses the equalities $\nabla_\xi\,\xi=0$ and \eqref{E-R1}, and is the same as for \cite[Lemma 3.1]{G-D-2020}.
\end{proof}

\begin{lemma}[see, for example, \cite{G-D-2020}]\label{L-5.2}
Let $(M; g)$ be a Riemannian manifold and $f\in C^\infty(M)$. Then the following holds for every vector fields $\xi,{Y}$ on $M:$
\[
 \pounds_{\,\xi}(df\otimes df)({Y},\xi) = {Y}(\xi(f))\,\xi(f) + {Y}(f)\,\xi(\xi(f)) .
\]
\end{lemma}

Recall that the Ricci curvature of any $K$-contact manifold satisfies the following condition:
\begin{equation}\label{E-K-Ric-{X}}
 \Ric(\xi,{X}) = 0\quad ({X}\in{\cal D}).
\end{equation}

\begin{lemma}\label{L-5.3}
Let a weak $K$-contact manifold satisfy \eqref{E-K-Ric-{X}} and admit the generalized gradient Ricci soliton structure \eqref{E-gg-r-e}.
Then
\[
 \nabla_\xi \nabla f = (\lambda + 2\,c_2 n + c_2\tr\tilde Q)\,\xi - c_1\xi(f)\,\nabla f .
\]
\end{lemma}

\begin{proof}
This uses \eqref{E-R1b} and \eqref{E-gg-r-e} and is analogous to the proof of \cite[Lemma 3.3]{G-D-2020}.
By \eqref{E-R1b} and \eqref{E-K-Ric-{X}} we get
\begin{equation}\label{E-Gh-3.4}
 \lambda\,\eta({Y}) +c_2\Ric(\xi,{Y}) = (\lambda + 2\,c_2 n + c_2\tr\tilde Q)\,\eta({Y}).
\end{equation}
Using \eqref{E-gg-r-e} and \eqref{E-Gh-3.4}, we get
\begin{equation}\label{E-Gh-3.5}
 {\rm Hess}_f(\xi,{Y}) = -c_1\xi(f)\,g(\nabla f, {Y}) + (\lambda + 2\,c_2 n + c_2\tr\tilde Q)\,\eta({Y}).
\end{equation}
Thus, \eqref{E-Gh-3.5} and the condition \eqref{E-gg-r-e} for the Hessian complete the proof.
\end{proof}

The next theorem generalizes \cite[Theorem 3.1]{G-D-2020}.

\begin{theorem}\label{T-5.1}
Let a weak $K$-contact manifold with $\tr\tilde Q=const$ satisfy the generalized gradient Ricci soliton equation \eqref{E-gg-r-e}
with $c_1(\lambda + 2\,c_2 n + c_2\tr\tilde Q)\ne -1$. Suppose that condition \eqref{E-K-Ric-{X}} is true.
Then $f=const$. Furthermore, if $c_2\ne0$, then the manifold is an Einstein one.
\end{theorem}

\begin{proof} Let ${Y}\in{\cal D}$. Then by Lemma~\ref{L-5.1} with ${X}=\nabla f$, we obtain
\begin{equation}\label{E-G3.6}
 2\,(\pounds_{\,\xi}({\rm Hess}_f))({Y},\xi) = {Y}(f) + g(\nabla_\xi\nabla_\xi \nabla f, {Y}) + {Y} g(\nabla_\xi \nabla f, \xi).
\end{equation}
Using Lemma~\ref{L-5.3} in \eqref{E-G3.6} and the properties $\nabla_\xi\,\xi=0$ and $g(\xi,\xi)=1$, yields
\begin{eqnarray}\label{E-G3.7}
\nonumber
 && 2\,(\pounds_{\,\xi}({\rm Hess}_f))({Y},\xi) = {Y}(f) + (\lambda + 2\,c_2 n + c_2\tr\tilde Q)\,g(\nabla_\xi\,\xi, {Y}) \\
\nonumber
 && -c_1 g(\nabla_\xi (\xi(f)\nabla f), {Y}) + (\lambda + 2\,c_2 n + c_2\tr\tilde Q)\,{Y}(g(\xi,\xi)) - c_1 {Y}(\xi(f)^2) \\
 && = {Y}(f) -c_1 g(\nabla_\xi(\xi(f)\nabla f), {Y}) - c_1 {Y}(\xi(f)^2).
\end{eqnarray}
Using Lemma~\ref{L-5.3} with ${Y}\in{\cal D}$, from \eqref{E-G3.7} it follows that
\begin{equation}\label{E-G3.8}
 2\,(\pounds_{\,\xi}({\rm Hess}_f))({Y},\xi) = {Y}(f) -c_1 \xi(\xi(f))\,{Y}(f) + c_1^2\xi(f)^2 {Y}(f) - c_1 {Y}(\xi(f)^2).
\end{equation}
Since $\xi$ is a Killing vector field, thus $\pounds_{\,\xi}\,g=0$, this implies $\pounds_{\,\xi}\Ric=0$.
Using the above fact and applying the Lie derivative to equation \eqref{E-gg-r-e}, gives
\begin{equation}\label{E-G3.9}
 2\,(\pounds_{\,\xi}({\rm Hess}_f))({Y},\xi) = -2\,c_1(\pounds_{\,\xi}(df\otimes df))({Y},\xi).
\end{equation}
Using \eqref{E-G3.8}, \eqref{E-G3.9} and Lemma~\ref{L-5.2}, we obtain
\begin{equation}\label{E-G3.10}
 {Y}(f)\big(1 + c_1 \xi(\xi(f)) + c_1^2\,\xi(f)^2 \big) = 0.
\end{equation}
By Lemma~\ref{L-5.3}, we get
\begin{equation}\label{E-G3.11}
 c_1 \xi(\xi(f)) = c_1\,\xi(g(\xi, \nabla f)) = c_1 g(\xi, \nabla_\xi\nabla f)
 = c_1(\lambda + 2\,c_2 n + c_2\tr\tilde Q) - c_1^2\,\xi(f)^2.
\end{equation}
Using \eqref{E-G3.10} in \eqref{E-G3.11}, we get ${Y}(f)(c_1(\lambda + 2\,c_2 n + c_2\tr\tilde Q)+1)=0$.
This implies ${Y}(f)=0$ provided by $c_1(\lambda + 2\,c_2 n + c_2\tr\tilde Q)+1 \ne 0$.
Hence, $\nabla f$ is parallel to $\xi$.
Taking the covariant derivative of $\nabla f = \xi(f)\,\xi$ and using \eqref{E-30}, we obtain
\[
 g(\nabla_{Z}\,\nabla f, {Y})= {Z}(\xi(f))\,\eta({Y}) -\xi(f)\,g(\varphi {Z}, {Y}),\quad {Z},{Y}\in\mathfrak{X}_M.
\]
From this, by symmetry of ${\rm Hess}_f$, i.e., $g(\nabla_{Z}\,\nabla f, {Y})=g(\nabla_{Y}\,\nabla f, {Z})$,
we get $\xi(f)\,g(\varphi {Z}, {Y})=0$.
For ${Y}=\varphi {Z}$ for some ${Z}\ne0$,  since $g(\varphi{Z}, \varphi{Z})>0$, we get $\xi(f)=0$; so $\nabla f=0$, i.e. $f=const$.
Thus, from \eqref{E-gg-r-e} and $c_2\ne0$ we conclude that the manifold is an Einstein manifold.
\end{proof}

\begin{remark}\rm
The following generalization of gradient Ricci soliton equation was given in \cite{CZB-2022}:
\begin{equation}\label{E-gg-r-e2}
 {\rm Hess}_{f_1} = -c_1 df_2\otimes df_2 + c_2\Ric + \lambda\,g
\end{equation}
for some functions $f_1,f_2\in C^\infty(M)$ and real $c_1, c_2$ and $\lambda$. For $f_1=f_2$, \eqref{E-gg-r-e2} reduces to \eqref{E-gg-r-e}.

Let a weak $K$-contact manifold with $\tr\tilde Q=const$ satisfy \eqref{E-K-Ric-{X}} and admit the generalized Ricci soliton structure \eqref{E-gg-r-e2} with $c_1(\lambda + 2\,c_2 n + c_2\tr\tilde Q)\ne -1$.
Then similarly to Lemma~\ref{L-5.3}, we get
\begin{equation}\label{E-gg-r-e2c}
 \nabla_\xi \nabla f_1 = (\lambda + 2\,c_2 n + c_2\tr\tilde Q)\,\xi - c_1\xi(f_2)\,\nabla f_2 .
\end{equation}
Using \eqref{E-gg-r-e2c} and Lemmas~\ref{L-5.1} and \ref{L-5.2}, and slightly modifying the proof of Theorem~\ref{T-5.1},
we find that the vector field $\nabla f$ is parallel to $\xi$, where $f=f_1 +c_1(\lambda + 2\,c_2 n + c_2\tr\tilde Q)\,f_2$.
Thus $df=0$, i.e.,
\[
 df_1 = -c_1(\lambda + 2\,c_2 n + c_2\tr\tilde Q)\,df_2.
\]
Using this in \eqref{E-gg-r-e2} and denoting $a:=\lambda + 2\,c_2 n + c_2\tr\tilde Q$,
we get
\begin{equation}\label{E-gg-r-e2b}
 -c_1 a\,{\rm Hess}_{f_2} = -c_1\,df_2\otimes df_2 + c_2\Ric + \lambda\,g.
\end{equation}
Then we obtain the following assertion (with three cases) that generalizes Theorem~\ref{T-5.1}.

1. If $c_1 a\ne0$, then \eqref{E-gg-r-e2b} reduces to
\begin{equation*}
 {\rm Hess}_{f_2} = \frac1a\,df_2\otimes df_2 - \frac{c_2}{c_1 a}\,\Ric - \frac\lambda{c_1 a}\,g.
\end{equation*}
By Theorem~\ref{T-5.1}, if $c_1 a\ne-1$, then $f_2=const$; moreover, if $c_2\ne0$, then $(M,g)$ is an Einstein manifold.

2. If $a=0$ and $c_1\ne0$, then \eqref{E-gg-r-e2b} reduces to
\begin{equation*}
 0 = c_2\Ric -c_1\,df_2\otimes df_2 + \lambda\,g.
\end{equation*}
If $c_2\ne0$ and $f_2\ne const$, then we get a gradient quasi Einstein manifold.
The concept of a quasi Einstein manifold was introduced
in \cite{CM-2000} by the condition
 $\Ric({X}, {Y}) = a\,g({X}, {Y}) + b\,\mu({X})\,\mu({Y})$
for all vector fields ${X},{Y}$, where $a$ and $b\ne0$ are real scalars, and $\mu$ is a 1-form of unit norm.

3. If $c_1=0$, then \eqref{E-gg-r-e2b} reduces to $0 = c_2\Ric + \lambda\,g$, and for $c_2\ne0$ we get an Einstein manifold.
\end{remark}

\section{Conclusion}

It is shown that the weak $K$-contact structure is a useful tool for studying unit Killing vector fields on Riemannian manifolds
and that some results for $K$-contact manifolds can be extended to the case of weak $K$-contact manifolds.
Inspired by \cite[Theorems~7.1 and 7.2]{blair2010riemannian}, the following question can be posed:
is the condition \eqref{E-Kmix}, or the weaker condition \eqref{E-R1b}, sufficient for a weak contact metric manifold to be weak $K$-contact?
To answer the question, some well known results for contact metric manifolds, such as \cite[Proposition~7.1]{blair2010riemannian}, must be generalized for weak contact metric manifolds.
In conclusion, we pose the following questions (about ``weak" analogues of results mentioned in \cite[Remark~3.2]{G-D-2020}):
is a compact weak $K$-contact Einstein mani\-fold a~Sasakian manifold?
Thus, is a compact weak $K$-contact mani\-fold admitting generalized Ricci soliton structure a~Sasakian manifold?


\baselineskip=12pt

\begin{thebibliography}{99}

\bibitem{blair2010riemannian}
     \newblock D.E. Blair,
     \newblock \emph{Riemannian geometry of contact and symplectic manifolds},
     \newblock Second edition,  Springer-Verlag, New York, 2010.

\bibitem{Blairsurvey}
     \newblock  D.E. Blair, 
     \newblock A survey of Riemannian contact geometry,
     \newblock \emph{Complex Manifolds}, \textbf{6} (2019), 31--64.

\bibitem{B-N-2008}
     \newblock  V.N. Berestovskij, and Yu.G. Nikonorov,
     \newblock Killing vector fields of constant length on Riemannian manifolds,
     \newblock \emph{Sib. Math. J.}, \textbf{49:3} (2008), 395--407.

\bibitem{CM-2000}
     \newblock  M.C. Chaki, and R.K. Maity,
     \newblock On quasi Einstein manifolds,
     \newblock \emph{Publ. Math. Debrecen}, \textbf{57} (2000), 297--306.

\bibitem{D-B-2021}
     \newblock  S. Deshmukh, and O. Belova,
     \newblock On Killing vector fields on Riemannian manifolds,
     \newblock \emph{Mathematics}, \textbf{9, 259} (2021), 1--17.

\bibitem{G-D-2020}
     \newblock  G. Ghosh, and U.C. De,
     \newblock Generalized Ricci soliton on $K$-contact manifolds,
     \newblock \emph{Math. Sci. Appl. E-Notes}, \textbf{8} (2020), 165--169.


\bibitem{KN-69}
S. Kobayashi, K. Nomizu, \textit{Foundations of differential geometry}, Vols. I, II, USA, Interscience Publishers, New York--London--Sydney, 1963, 1969.


\bibitem{CZB-2022}
A. Mohammed Cherif, K. Zegga and G. Beldjilali, On the generalised Ricci solitons and Sasakian manifolds.
Communications in Mathematics, 30, Issue 1 (2022) 119--123.


\bibitem{N-2021}
V. N. Berestovskij, Yu. G. Nikonorov,
Killing vector fields of constant length on Riemannian manifolds.
\emph{Sib. Math. J.} \textbf{49}, No. 3 (2008), 395--407.

\bibitem{N-R-2016}
     \newblock  P. Nurowski, and M. Randall,
     \newblock Generalised Ricci solitons,
     \newblock \emph{J. Geom. Anal.}, \textbf{26} (2016), 1280--1345.

\bibitem{RovP-arxiv}
\newblock D.S. Patra and V. Rovenski,
\newblock On the rigidity of the Sasakian structure and characterization of cosymplectic manifolds,
\newblock preprint, arXiv:math/2203.04597.



\bibitem{RWo-2}
     \newblock  V. Rovenski, and R. Wolak,
     \newblock New metric structures on $\mathfrak{g}$-foliations,
     \newblock \emph{Indagationes Mathematicae}, \textbf{33} (2022), 518--532.

\bibitem{YK-1985}
     \newblock K. Yano, and M. Kon,
     \newblock \emph{Structures on Manifolds},
     \newblock Vol. 3 of Series in Pure Math. World Scientific Publ. Co., Singapore, 1985.

\end{thebibliography}
\end{document}